\documentclass{amsart}
\usepackage{amsmath}
\usepackage{amsthm}
\usepackage{amssymb}
\usepackage{color}
\usepackage{enumerate}

\newtheorem{thm}{Theorem}
\newtheorem{lem}[thm]{Lemma}

\newtheorem{cor}[thm]{Corollary}
\newtheorem{defi}[thm]{Definition}

\newtheorem{rk}[thm]{Remark}

\newcommand{\rr}{{\mathbb{R}}}

\newcommand{\cF}{{\mathcal{F}}}
\newcommand{\cL}{{\mathcal{L}}}
\newcommand{\cP}{{\mathcal{P}}}
\newcommand{\e}{\epsilon}
\newcommand{\vip}{\vskip.12cm}
\newcommand{\indiq}{{\bf 1}}
\newcommand{\E}{\mathbb{E}}
\newcommand{\PR}{\mathbb{P}}
\newcommand{\intot}{\int_0^t }
\newcommand{\dd}{{\rm d}}
\newcommand{\sg}{{\rm sg}}
\newcommand{\hX}{\hat X}
\newcommand{\hV}{\hat V}
\newcommand{\bA}{\bar A}
\newcommand{\btau}{\bar \tau}

\begin{document}

\title[Critical kinetic Fokker-Planck equations]
{One dimensional critical Kinetic Fokker-Planck equations, Bessel and stable processes}

\author{Nicolas Fournier and Camille Tardif}

\address{Sorbonne Universit\'e - LPSM, Campus Pierre et Marie Curie, 
Case courrier 158, 4 place Jussieu, 75252 PARIS CEDEX 05,
{\tt nicolas.fournier@sorbonne-universite.fr, camille.tardif@sorbonne-universite.fr}.}

\thanks{We warmly thank Quentin Berger for illuminating discussions. This research was supported 
by the French  ANR-17- CE40-0030 EFI}



\begin{abstract}
We consider a particle moving in one dimension, its 
velocity being a reversible diffusion process,
with constant diffusion coefficient, of which the invariant measure behaves like $(1+|v|)^{-\beta}$ for 
some $\beta>0$.
We prove that, under a suitable
rescaling, the position process resembles a Brownian motion if $\beta\geq 5$,
a stable process if $\beta\in [1,5)$ and an integrated symmetric Bessel process if $\beta\in (0,1)$.
The critical cases $\beta=1$ and $\beta=5$ require special rescalings.
We recover some results of \cite{np,cnp,lp} and \cite{bak} with an alternative approach.
\end{abstract}

\maketitle

\section{Introduction and results}
\setcounter{equation}{0}

We consider the following one-dimensional stochastic kinetic model:
\begin{equation}\label{eds}
V_t = V_0 + B_t - \frac{\beta}{2} \intot F(V_s) \dd s\quad \hbox{and} \quad X_t=X_0+\intot V_s \dd s.
\end{equation}
Here $(B_t)_{t\geq 0}$ is a Brownian motion independent of the initial condition $(X_0,V_0)$.
We assume that $\beta>0$ and that the force is of the form
\begin{equation}\label{condi}
F=-\frac{\Theta'}{\Theta}, \hbox{ for some even $\Theta:\rr\mapsto (0,\infty)$ 
of class $C^2$ satisfying $\lim_{|v|\to \infty} |v|\Theta(v)=1$.}
\end{equation} 
The typical example we have in mind is $F(v)=v/(1+v^2)$, which corresponds to
$\Theta(v)=(1+v^2)^{-1/2}$.

\vip

The drift $F$ being $C^1$, \eqref{eds} classically has a unique (possibly local) strong solution, 
and we will see that it is global.
An invariant measure $\mu_\beta$ of the velocity process $(V_t)_{t\geq 0}$ must solve 
$\frac12 \mu_\beta''+\frac\beta2(F \mu_\beta)'=0$ in the sense of distributions.
The unique (up to constants) solution is given by 
$$
\mu_\beta(\dd v)=c_\beta[\Theta(v)]^\beta\dd v,
$$
and we choose $c_\beta^{-1}=\int_\rr [\Theta(v)]^\beta \dd v<\infty$ if $\beta>1$ and $c_\beta=1$ 
if $\beta \in (0,1]$.

\subsection{Main result}\label{mrc}

Our goal is to describe the large time behavior of the position process $(X_t)_{t\geq 0}$.
For each $\beta\geq 1$, we define the constant $\sigma_\beta>0$ as follows:

\vip

$\bullet$ $\sigma_\beta^2=8 c_\beta\int_0^\infty \Theta^{-\beta}(v)[\int_v^\infty u\Theta^\beta(u)\dd u]^2 \dd v$ if 
$\beta>5$,

\vip

$\bullet$ $\sigma_5^2=4c_5 /27$,

\vip

$\bullet$ $\sigma_\beta^\alpha = 3^{1-2\alpha} 2^{\alpha-1} c_\beta \pi / [(\Gamma(\alpha))^2\sin(\pi\alpha/2)]$, 
where $\alpha=(\beta+1)/3$, if $\beta \in (1,5)$,

\vip

$\bullet$ $\sigma_1^{2/3}= 2^{2/3}3^{-5/6} \pi/[\Gamma(2/3)]^2$.

\vip

Consider a family $((Z^\e_{t})_{t\geq 0})_{\e\geq 0}$ of processes.
We write $(Z^\e_{t})_{t\geq 0}\stackrel{f.d.}\longrightarrow (Z_t^0)_{t\geq 0}$ if for all 
finite subset $S\subset [0,\infty)$ the vector $(Z^\e_{t})_{t\in S}$ goes in law to $(Z_t^0)_{t\in S}$
as $\e\to 0$.
We write $(Z^\e_{t})_{t\geq 0}\stackrel{d}\longrightarrow (Z_t)_{t\geq 0}$ if the convergence in law holds 
in the usual sense of continuous processes. Here is our main result.

\begin{thm}\label{mr}
Fix $\beta>0$ and consider the solution $(X_t,V_t)_{t\geq 0}$ to \eqref{eds}.
Let $(B_t)_{t\geq 0}$ be a Brownian motion, let $(S^{(\alpha)}_t)_{t\geq 0}$ be a symmetric stable process 
with index $\alpha \in (0,2)$
such that $\E[\exp(i\xi S^{(\alpha)}_t)]=\exp(-t |\xi|^\alpha)$ and let $(U^{(\delta)}_t)_{t\geq 0}$ be
a symmetric Bessel process of dimension $\delta \in (0,1)$, see Definition \ref{bs}.

\vip
(a) If $\beta>5$, $(\e^{1/2} X_{t/\e} )_{t\geq 0} \stackrel{f.d.}\longrightarrow (\sigma_\beta B_t)_{t\geq 0}$.

\vip

(b) If $\beta = 5$, $(\e^{1/2}|\log\e|^{-1/2} X_{t/\e} )_{t\geq 0}\stackrel{f.d.}
\longrightarrow (\sigma_5 B_t)_{t\geq 0}$.

\vip

(c) If $\beta \in (1,5)$, $(\e^{1/\alpha} X_{t/\e} )_{t\geq 0}\stackrel{f.d.}
\longrightarrow (\sigma_\beta S^{(\alpha)}_t)_{t\geq 0}$, where $\alpha=(\beta+1)/3$.

\vip

(d) If $\beta=1$, $(|\e \log \e|^{3/2} X_{t/\e} )_{t\geq 0}\stackrel{f.d.}
\longrightarrow (\sigma_1 S^{(2/3)}_t)_{t\geq 0}$.

\vip

(e) If $\beta \in (0,1)$,  $(\e^{1/2}V_{t/\e},\e^{3/2} X_{t/\e})_{t\geq 0}\stackrel{d}\longrightarrow 
(U^{(1-\beta)}_t,\int_0^t U^{(1-\beta)}_s\dd s)_{t\geq 0}$.
\end{thm}

We will deduce the following decoupling between the position and the velocity.

\begin{cor}\label{coco}
Fix $\beta>1$, adopt the same notation as in Theorem \ref{mr} and consider a $\mu_\beta$-distributed
random variable $\bar V$ independent of everything else.\vip

(a) If $\beta>5$, for each $t>0$, $(\e^{1/2} X_{t/\e},V_{t/\e}) \stackrel{d}\longrightarrow 
(\sigma_\beta B_t,\bar V)$.

\vip

(b) If $\beta = 5$, for each $t>0$, $(\e^{1/2}|\log\e|^{-1/2} X_{t/\e}, V_{t/\e})\stackrel{d}
\longrightarrow (\sigma_5 B_t, \bar V)$.

\vip

(c) If $\beta \in (1,5)$, for each $t>0$, $(\e^{1/\alpha} X_{t/\e},V_{t/\e}) \stackrel{d}
\longrightarrow (\sigma_\beta S^{(\alpha)}_t,\bar V)$, where $\alpha=(\beta+1)/3$.
\end{cor}

\subsection{The Fokker-Planck equation}

Consider a solution $(X_t,V_t)_{t\geq 0}$ to \eqref{eds} and denote, for each $t\geq 0$, 
by $f_t\in \cP(\rr\times\rr)$ the law of $(X_t,V_t)$.
We write the paragraph doing as if 
$f_t$ had a density, that we still 
denote by $f_t$, for each $t>0$. This probably always holds true, but this is not the subject of the paper.
Of course, everything can be written, with heavier notation, when $f_t$ is (possibly) singular.
By the It\^o formula, $(f_t)_{t\geq 0}$ is a weak solution to the Fokker-Planck equation
\begin{equation}\label{fpe}
\partial_t f_t(x,v) + v \partial_xf_t(x,v)= \frac 12 \big(\partial_{vv}f_t(x,v) 
+ \beta \partial_v[ F(v) f_t(x,v)]\big).
\end{equation}

Corollary \ref{coco}-(a) tells us that, when $\beta>5$, in a weak sense, 
$\e^{-1/2} f_{\e^{-1}t}(\e^{-1/2}x,v)$ goes, as $\e\to 0$, to 
$g_t(x)\mu_\beta(v)$, where $g_t(x)$ is the centered Gaussian density with variance $\sigma_\beta^2t$. 
%This had already been proved, in any
%dimension, for the force $F(v)=v/(1+v^2)$, by Nasreddine and Puel \cite{np}.

\vip

If $\beta=5$, by Corollary \ref{coco}-(b), still in a weak sense, 
$\e^{-1/2}|\log\e|^{1/2} f_{\e^{-1}t}(\e^{-1/2}|\log\e|^{1/2}x,v)$ 
goes, as $\e\to 0$, to $g_t(x)\mu_\beta(v)$, where $g_t(x)$ is 
the centered Gaussian density with variance $\sigma_5^2t$. 
%This had already been checked, in any
%dimension, for $F(v)=v/(1+|v|^2)$, by Cattiaux, Nasreddine and Puel \cite{cnp}.

\vip

If $\beta\in (1,5)$, by Corollary \ref{coco}-(c), setting $\alpha=(\beta+1)/3$,
$\e^{-1/\alpha}f_{\e^{-1}t}(\e^{-1/\alpha}x,v)$ 
goes, as $\e\to 0$, to $g_t(x)\mu_\beta(v)$, where $g_t(x)$ is 
the symmetric stable law characterized by its Fourier transform
$\int_\rr g_t(x)e^{i\xi x}\dd x = \exp(- t |\sigma_\beta \xi |^\alpha)$.
%This had already been proved, in dimension $1$ and
%for the force $F(v)=v/(1+|v|^2)$, by Lebeau and Puel \cite{lp}.

\vip

If $\beta=1$, by Theorem \ref{mr}-(d), setting $\rho_t(x)=\int_\rr f_t(x,v)\dd v$, 
$|\e \log\e|^{-3/2} \rho_{\e^{-1}t}(|\e\log\e|^{-3/2}x)$ goes, in a weak sense, to $g_t(x)$, where $g_t(x)$ is
characterized by $\int_\rr g_t(x)e^{i\xi x}\dd x = \exp(- t |\sigma_1 \xi |^{2/3})$.

\vip

Finally, when $\beta \in (0,1)$,  Theorem \ref{mr}-(e) tells us that
$\e^{-2}f_{\e^{-1}t}(\e^{-3/2} x, \e^{-1/2} v)$ converges to $h_t(x,v)$, the density of the law
of $(U^{(1-\beta)}_t,\int_0^t U^{(1-\beta)}_s\dd s)$.
It should be that $(h_t)_{t\geq 0}$ is a symmetric solution, in some very weak sense, 
of \eqref{fpe}, with $h_0=\delta_{(0,0)}$ and $F(v)=v^{-1}\indiq_{\{v \neq 0\}}$.

\subsection{Motivation and references}

Performing a space-time rescaling in Boltzmann-like equations to get
some diffusion limit for the position process is an old and
classical subject, see for example Larsen and Keller 
\cite{lk}, Bensoussans, Lions and Papanicolaou \cite{blp}.
More recently, Bodineau, Gallagher and Saint-Raymond \cite{bgsr}
obtained some Brownian motion, starting from a system of hard spheres.

\vip

However, anomalous diffusion often arises in physics, and many works show how to modify
the collision kernel in some Boltzmann-like linear equations to get some fractional diffusion limit
(i.e. a symmetric stable jumping position process).
One can e.g. linearize the Boltzmann equation around a fat tail equilibrium or consider some
{\it ad hoc} cross section.
This was initiated by 
Mischler, Mouhot and Mellet \cite{mmm}, with close links to
the earlier work of Milton, Komorowski and Olla \cite{mko} on Markov chains. This was continued by 
Mellet \cite{m}, Ben Abdallah, Mellet and Puel \cite{bmp,bmp2} and others.

\vip

The Fokker-Planck equation is often used in physics to approximate Boltzmann-like equations, because it is
generally more simple to tackle.
When the repealing/friction force $F$ 
is strong enough, the velocities have an invariant measure with fast decay. 
When $F(v)=v$, the diffusion limit for the position process 
was then predicted by Langevin in \cite{l}. For generalizations and recent analysis
see e.g. Cattiaux, Chafa\"i and Guillin \cite{ccg}.

\vip

The only way to hope for some anomalous diffusion limit, for a Fokker-Planck toy model like \eqref{eds}, 
is to choose
the repealing force in such a way that the invariant measure of the velocity process has a fat tail. 
Then one realizes that we have to choose $F$ behaving like $F(v)\sim 1/v$ as $|v|\to \infty$,
and the most natural choice is $F(v)=v/(1+v^2)$. This case has been studied 
by Nasreddine and Puel \cite{np} (normal diffusion, $\beta>5$, in any dimension),  
Cattiaux, Nasreddine and Puel \cite{cnp}
(critical case $\beta=5$, in any dimension) 
and Lebeau and Puel \cite{lp} (anomalous diffusion, $\beta \in (1,5)$, in one dimension), the
case $\beta \in (0,1]$ being left open.

\vip

For this model, there is no (weighted) Poincar\'e inequality for the velocity process, so that 
this process converges quite slowly to equilibrium.
This seems to be an issue when $\beta\leq 5$. 
Some more complicated inequalities, see Cattiaux, Gozlan, Guillin and Roberto \cite{cggr}, are used in
\cite{cnp,lp}. 

\vip

The anomalous diffusion case $\beta \in (1,5)$ for \eqref{eds} is rather difficult to treat,
in comparison to above cited works 
\cite{mmm,m,bmp,bmp2} on Boltzmann-like equations. In particular, while the 
stable index $\alpha$ is more or less prescribed from the beginning in \cite{mmm,m,bmp,bmp2}, it is
a rather mysterious function of $\beta$ in the present case. The paper \cite{lp} relies on a
deep spectral analysis, making a wide use of special functions. From the probabilistic point of view,
the models studied in \cite{mmm,m,bmp,bmp2} correspond to jumping velocity processes with fat tail L\'evy measures,
so that stable L\'evy limits are quite natural.

\vip

In a somewhat different perspective, physicists discovered that atoms, when cold by a laser,
diffuse anomalously, like \emph{L\'evy walks}. See
Castin, Dalibard and Cohen-Tannoudji \cite{cdct}, Sagi, Brook, Almog and Davidson \cite{sbad} and 
Marksteiner, Ellinger and Zoller \cite{mez}. 
A theoretical study has been proposed by Barkai, Aghion and Kessler \cite{bak} 
(see also
and Hirschberg, Mukamel and Sch\"utz \cite{hms}). They precisely model the motion of atoms by
\eqref{eds} with the force $F(v)=v/(1+v^2)$ induced by the laser field.
They prove, with a quite high level of rigor, the results of Theorem \ref{mr}, excluding the critical cases
and stating point (e), that they call Obukhov-Richardson phase, differently.

\vip

Let us say a word of the proof of \cite{bak}. 
One easily gets convinced that the velocity process behaves like a (symmetrized) 
Bessel process with dimension
$\delta=1-\beta$ when far away from $0$, simply because $F(v)\simeq v^{-1}$. But, even when $\delta \leq 0$, 
this velocity process is not stuck when it reaches $0$. Consequently, the position process
can be approximated by a sum of i.i.d. (signed) areas of excursions of Bessel processes with 
dimension $\delta$, with a
sense to be precised when $\delta\leq 0$ (consider the area under the Bessel process until it reaches $0$, 
when starting from $v>0$, and let $v\to 0$ with in a clever way). Using some explicit computations relying on
modified Bessel functions, they show that this area has a fat tail distribution, with a density
decaying at infinity like $x^{-1-\alpha}$, where 
$\alpha =(\beta+1)/3$. Hence if $\beta>5$, then $\alpha>2$ and 
this area has a second order moment, whence the classical central
limit theorem applies, and normal diffusion occurs. If now $\alpha<2$, one has to use a {\it stable}
limit theorem, and anomalous diffusion arises.

\vip

Actually, when $\beta < 1$, the {\it length} of the Bessel excursion is no more integrable and the above proof 
breaks down. Barkai, Aghion and Kessler \cite{bak} introduce some Bessel bridges and 
handle some tedious explicit computations. But the situation
is actually much easier since, at least at the informal level, the Bessel process with dimension
$\delta=1-\beta>0$ is not stuck when it reaches $0$, so that one can simply approximate the
velocity process by a {\it true} (symmetrized) Bessel process with dimension $\delta$.

\subsection{Goal and strategy}
Our initial goal was to completely formalize the arguments of \cite{bak}, in order to provide
a probabilistic proof, including the critical cases, of the results of \cite{np,cnp,lp}.
We found another way, which is more qualitative and even more probabilistic,
making use of the connections (or similarities) between Bessel and stable processes, see Section \ref{rrr}. 
We provide a concise proof, that moreover allows us to deal with
general forces of the form \eqref{condi}.

\vip

The core of the paper (when $\beta\leq 5$, which is the most interesting case) 
consists in making precise the following {\it informal} arguments.
For $(W_t)_{t\geq 0}$ a Brownian motion and for $\tau_t$ the inverse of the time change $A_t=(\beta+1)^{-2}\int_0^t 
|W_s|^{-2\beta/(\beta+1)}\dd s$, the process $Y_t=W_{\tau_t}$ should classically solve, see e.g.  
Revuz-Yor \cite[Proposition 1.13 page 373]{ry}, 
$Y_t=(\beta+1)\int_0^t |Y_s|^{\beta/(\beta+1)}
\dd B_s$, for some other Brownian motion $(B_t)_{t\geq 0}$. Hence, still informally, $V_t=\sg(Y_t)|Y_t|^{1/(\beta+1)}$
should solve, by the It\^o formula, 
\begin{equation}\label{tb}
V_t=B_t -(\beta/2)\intot \sg(V_s)|V_s|^{-1}\dd s, 
\end{equation}
This is a rough version of \eqref{eds} with $F(v)=\sg(v)|v|^{-1}$ and we admit that,
after rescaling, this describes some large time behavior of the true solution to \eqref{eds}.

\vip

We recognize in \eqref{tb} a symmetric version of the SDE for a Bessel 
process of dimension $\delta=1-\beta$.

\vip

If $\beta\in (0,1)$, i.e. $\delta>0$, such a (symmetric) Bessel process is well-defined and non-trivial,
see also Definition \ref{bs} below. Thus 
Theorem \ref{mr}-(e) is not surprising.

\vip

If $\beta\in (1,5)$, i.e. $\delta\leq 0$, it is well-known that $V_t$ will remain stuck at $0$.
But it actually appears that $A_t$ is infinite and, in some sense to be precised, 
proportional to the local time $L^0_t$ of $(W_t)_{t\geq 0}$. Hence, up to correct rescaling, 
$X_t=\int_0^{t} V_s\dd s = \int_0^{t} \sg(W_{\tau_s})|W_{\tau_s}|^{1/(\beta+1)}\dd s
= \int_0^{\tau_t} \sg(W_{s})|W_{s}|^{1/(\beta+1)}\dd A_s=(\beta+1)^{-2}\int_0^{\tau_t} \sg(W_{s})|W_{s}|^{(1-2\beta)/(\beta+1)}
\dd s$. Since $(\tau_t)_{t\geq 0}$ is proportional to the inverse of the local time
of $(W_t)_{t\geq 0}$, we know from Biane-Yor \cite{by} that 
 $(X_t)_{t\geq 0}$ is an $\alpha$-stable L\'evy process, 
with $\alpha=(\beta+1)/3$. See Theorem \ref{ss} below for a precise statement and a few explanations.
All this is completely informal, in particular observe that we always have $W_{\tau_t}=0$, so that
the equality $\int_0^{t} \sg(W_{\tau_s})|W_{\tau_s}|^{1/(\beta+1)}\dd s
= \int_0^{\tau_t} \sg(W_{s})|W_{s}|^{1/(\beta+1)}\dd A_s$ is far from being fully justified.

\vip

The proof of Theorem \ref{ss} does not require deep computations involving specials functions,
unless one wants to know the value of the diffusion constant. This is why we say that 
our proof is {\it qualitative}.

\section{Stable and Bessel processes}\label{rrr}
In the whole paper,
we denote by $\sg$ the sign function with the convention that $\sg(0)=0$.

\subsection{Stable processes}
We will use the following theorem that we found in Biane-Yor \cite{by}.
Very similar results were already present in It\^o-McKean \cite[page 226]{imk}
and Jeulin-Yor \cite{jy} when $\alpha\in (0,1)$.

\begin{thm}\label{ss}
Fix $\alpha \in (0,2)$. Consider a Brownian motion $(W_t)_{t\geq 0}$, its local time $(L_t^0)_{t\geq 0}$ at $0$
and its right-continuous generalized inverse $\tau_t=\inf\{u\geq 0 : L^0_u > t\}$. For $\eta>0$, let
$K_t^\eta= \intot \sg(W_s)|W_s|^{1/\alpha - 2} \indiq_{\{|W_s|\geq \eta\}}\dd s$.
Then $(K^\eta_t)_{t\geq 0}$ a.s. converges to some process $(K_t)_{t\geq 0}$ as $\eta\to 0$,
and $(K_{\tau_t})_{t\geq 0}$ is a symmetric $\alpha$-stable process such that
$$
\E[\exp(i\xi K_{\tau_t})]=\exp(-\kappa_\alpha t |\xi|^\alpha),
\quad \hbox{where} \quad
\kappa_\alpha=\frac{2^\alpha \pi \alpha^{2\alpha}}{2\alpha[\Gamma(\alpha)]^2 \sin(\pi\alpha/2)}.
$$
\end{thm}

When $\alpha \in (0,1)$, we simply have $K_t= \intot \sg(W_s)|W_s|^{1/\alpha - 2}\dd s$, since this integral 
is a.s. absolutely convergent.
Theorem \ref{ss} is very natural and easy to verify, as far as we are
not interested in the exact value of $\kappa_\alpha$. Indeed, $\tau_t$ is a stopping-time,
with $W_{\tau_t}=0$, for each $t\geq 0$. Hence
the strong Markov property implies that 
the process $Z^\varphi_{t}=\int_0^{\tau_t}\varphi(W_s)\dd s$
is L\'evy, for any reasonable function $\varphi:\rr\mapsto\rr$.
It is furthermore of course symmetric if $\varphi$ is odd. If finally one wants $Z^\varphi_{t}$
to satisfy the scaling property of $\alpha$-stable processes, i.e.
$Z^\varphi_{t}\stackrel{d}=c^{-1/\alpha}Z^\varphi_{c t}$, 
there is no choice for $\varphi$:
it has to be $\varphi(z)=\sg(z)|z|^{1/\alpha - 2}$, recall that 
$((\tau_t)_{t\geq 0}, (W_t)_{t\geq 0})\stackrel{d}=(c^{-2}\tau_{c t},c^{-1}W_{c^2t})$.
All this perfectly holds true when $\alpha\in (0,1)$,
but there are some small difficulties when $\alpha\in [1,2)$ because the integral
$\intot \sg(W_s)|W_s|^{1/\alpha - 2}\dd s$ is not absolutely convergent. The 
computation of $\kappa_\alpha$ is rather tedious and involves special functions.

\subsection{Bessel processes}

Bessel processes are studied in details in Revuz-Yor \cite[Chapter XI]{ry}. 
These are nonnegative processes,
and we need a signed symmetric version.
Roughly, we would like to take a Bessel process and to change the sign of each excursion, independently, 
with probability 
$1/2$. Inspired by Donati-Roynette-Vallois-Yor \cite{drvy}, we will rather use the following 
(equivalent) definition.

\begin{defi}\label{bs}
Let $\delta \in (0,2)$. Consider a Brownian motion $(W_t)_{t\geq 0}$, introduce the time-change
$\bA_t=(2-\delta)^{-2}\intot |W_s|^{-2(1-\delta)/(2-\delta)} \dd s$ and its inverse $(\btau_t)_{t\geq 0}$.
We set $U^{(\delta)}_t=\sg(W_{\btau_t}) |W_{\btau_t}|^{1/(2-\delta)}$ and say that $(U^{(\delta)}_t)_{t\geq 0}$ is a 
symmetric Bessel process
with dimension $\delta$.
\end{defi}

Since $2(1-\delta)/(2-\delta)<1$, $\E[\bA_t]<\infty$ for all $t\geq0$.
The map $t\mapsto \bA_t$ is a.s. continuous, strictly increasing and $\bA_\infty=\infty$
by recurrence of $(W_t)_{t\geq 0}$,
so that $(\btau_t)_{t\geq 0}$ is well-defined and continuous.
Also, $U_t^{(1)}=W_t$: the Brownian motion
is the symmetric version of the Bessel process of dimension $1$. 

\vip

To justify the terminology, let us mention that 
$(|U^{(\delta)}_t|)_{t\geq 0}$ is
a Bessel process with dimension $\delta$. Indeed, \cite[Corollary 2.2]{drvy} tells us that, for
$(R_t)_{t\geq 0}$ a Bessel process with dimension $\delta$, $R_t=|W_{C_t}|^{1/(2-\delta)}$
for some Brownian motion $(W_t)_{t\geq 0}$ and for $C_t=(2-\delta)^2\intot R_s^{2(1-\delta)}\dd s$.
But $C_t=\btau_t$, whence $R_t=|U^{(\delta)}_t|$,
because
$\bA_{C_t}=(2-\delta)^{-2}\int_0^{C_t}|W_s|^{-2(1-\delta)/(2-\delta)} \dd s=
\int_0^t |W_{C_u}|^{-2(1-\delta)/(2-\delta)} R_u^{2(1-\delta)}\dd u= t$.

\section{Proofs}

Here is the strategy of the proof. We first verify quickly that
the velocity process is well-defined for all times, regular and recurrent, and we
explain why it suffices to prove Theorem \ref{mr} when $X_0=V_0=0$.

\vip

In a second subsection, we (classically) check Theorem \ref{mr} in the normal
diffusive case $\beta>5$. 

\vip In Subsection \ref{ss3}, we introduce some functions $\Psi_\e:\rr\mapsto\rr$ and
$\sigma_\e:\rr \mapsto(0,\infty)$ such that, for $(W_t)_{t\geq 0}$ a Brownian motion and for
$(\tau^\e_t)_{t\geq 0}$ the inverse of $(A^\e_t)_{t\geq 0}$ defined by 
$A^\e_t=\intot [\sigma_\e(W_s)]^{-2}\dd s$, the processes $(V_{t/\e})_{t\geq 0}$
and $(\Psi_\e(W_{\tau^\e_t}))_{t\geq 0}$ have the same law.

\vip

In Subsection \ref{ss5}, we prove our main result when $\beta\in (0,1)$.
We first write $(\e^{1/2} V_{t/\e},\e^{3/2}X_{t/\e})_{t\geq 0}\stackrel{d}=(\e^{1/2} \Psi_\e(W_{\tau^\e_t}),
\int_0^t \e^{1/2} \Psi_\e(W_{\tau^\e_s}) \dd s)_{t\geq 0}$.
And we show that $\e^{1/2} \Psi_\e(z)$ resembles $\sg(z)|z|^{1/(1+\beta)}$
and that $\sigma_\e(z)$ resembles $(\beta+1)|z|^{\beta/(\beta+1)}$. 
Consequently, recalling Definition \ref{bs} (with $\delta=1-\beta$),
$A^\e_t\simeq \bA_t$, whence $\tau^\e_t\simeq \btau_t$, and $\e^{1/2} \Psi_\e(W_{\tau^\e_t})\simeq \sg(W_{\btau_t})
|W_{\btau_t}|^{1/(\beta+1)}=U^{{(1-\beta)}}_t$ as desired.

\vip

We study the case where $\beta\in [1,5]$ in Subsection \ref{ss6}. 
Up to logarithmic corrections when $\beta=1$ or $5$, we write 
$(\e^{1/\alpha}X_{t/\e})_{t\geq 0} \stackrel{d}=(\e^{1/\alpha-1}\!\intot \Psi_\e(W_{\tau^\e_s})\dd s)_{t\geq 0}
=(\e^{1/\alpha-1}\!\int_0^{\tau^\e_t} \Psi_\e(W_{s})[\sigma_\e(W_s)]^{-2}\dd s)_{t\geq 0}$. We then
recall Theorem \ref{ss} with $\alpha=(\beta+1)/3$, and we check that
$A^\e_t\simeq L^0_t$, whence $\tau^\e_t\simeq \tau_t$, and that 
$\e^{1/\alpha-1}\Psi_\e(z)[\sigma_\e(z)]^{-2}\simeq \sg(z)|z|^{(1-2\beta)/(1+\beta)}
=\sg(z)|z|^{1/\alpha-2}$. All in all, we deduce that $(\e^{1/\alpha}X_{t/\e})_{t\geq 0} \stackrel{d}\simeq 
(\int_0^{\tau_t} \sg(W_s)|W_s|^{1/\alpha-2}\dd s)_{t\geq 0}$,
which is a symmetric $\alpha$-stable process. 

\vip

In this last case $\beta \in [1,5]$, the situation is actually more complicated, there 
are some constants appearing everywhere.
Let us also mention that 
the case $\beta=5$ requires a special study, using both the standard method (as when $\beta>5$)
and some local time arguments.

\vip 

Finally, the last subsection is devoted to the proof of Corollary \ref{coco}.

\subsection{Preliminaries}

Let us first prove the following.

\begin{lem}\label{ci}
(a) The solution $(V_t)_{t\geq 0}$ to \eqref{eds} is global, regular and recurrent.

\vip

(b) There is $C>0$ such that, if $V_0=0$, $\E[V_t^2+|V_t|^{\beta+1}]\leq C (1+t)$ for all $t\geq 0$.

\vip

(c) If Theorem \ref{mr}-(a)-(b)-(c)-(d) holds when $X_0=V_0=0$ a.s., then it holds for any initial condition.

\vip

(d) To prove Theorem \ref{mr}-(e), it suffices to prove that $(\e^{1/2}V_{t/\e})_{t\geq 0}\stackrel{d}\longrightarrow 
(U^{(1-\beta)}_t)_{t\geq 0}$ when $V_0=0$.
\end{lem}

\begin{proof}
We first verify (a).
As we will see in Lemma \ref{ex} with $\e=a_\e=1$, the solution $(V_t)_{t\geq 0}$ to \eqref{eds}
with $V_0=0$ has the same law as
$(h^{-1}(W_{\tau_t}))_{t\geq 0}$, for some Brownian motion $(W_t)_{t\geq 0}$, some (random) continuous bijective 
time change 
$\tau_t:[0,\infty)\mapsto [0,\infty)$ and some continuous bijective function $h:\rr\mapsto \rr$. 
Hence $(V_t)_{t\geq 0}$
is non-exploding and thus global, and it is regular and recurrent (when starting from any initial condition).

\vip

We next prove (b). The even function $\ell(v)=2 \int_0^v \Theta^{-\beta}(x)\int_0^x \Theta^\beta(u)\dd u \dd x$ 
solves the Poisson equation $\ell''(v)-\beta F(v)\ell'(v)=2$, whence $\E[\ell(V_t)]=t$,
by the It\^o formula and since $\ell(V_0)=\ell(0)=0$.
Using \eqref{condi}, we see that there is a constant $c>0$ such that, as $|v|\to \infty$,
$\ell(v) \sim c |v|^{\beta+1}$ if $\beta>1$, $\ell(v) \sim c |v|^2 \log |v|$ if $\beta=1$, and 
$\ell(v)\sim c v^2$ if $\beta \in (0,1)$. Thus in any case, we can find a constant
$C$ such that $v^2+|v|^{\beta+1} \leq C(\ell(v)+1)$ for all $v\in \rr$, whence 
$\E[V_t^2+|V_t|^{\beta+1}]\leq C (1+\E[\ell(V_t)])=C(1+t)$.

\vip

We now check (c). Assume that for 
$\beta\geq 1$, Theorem \ref{mr} holds when starting from $(0,0)$ and consider the solution 
$(V_t,X_t)_{t\geq 0}$ to \eqref{eds}
starting from some $(V_0,X_0)$. We introduce $\tau=\inf\{t\geq 0 : V_t=0\}$, which is a.s. finite by recurrence.
Then $(\hV_t,\hX_t)=(V_{\tau+t},X_{\tau+t}-X_\tau)$ solves \eqref{eds}, starts from $(0,0)$, and is independent 
of $\tau$ by the strong Markov property. We thus know that $(v^{(\beta)}_\e \hX_{t/\e})_{t\geq 0}\stackrel{f.d.}\to 
(X^{(\beta)}_{t})_{t\geq 0}$, where $v^{(\beta)}_\e\to 0$ and $(X^{(\beta)}_{t})_{t\geq 0}$ are the rate and limit process
appearing in Theorem \ref{mr}.
We now prove that for each $t\geq 0$,
$v^{(\beta)}_\e |X_{t/\e}-\hX_{t/\e}|\to 0$ in probability, and this will complete the proof.
We introduce $D^1=|X_0|+\int_0^{2\tau}|V_s|\dd s$ and
$D^{2,\e}_t=\indiq_{\{t/\e\geq \tau\}}\int_{t/\e-\tau}^{t/\e}|\hV_s|\dd s$ and observe that
$|X_{t/\e}-\hX_{t/\e}| \leq D^1+D^{2,\e}_t$. Indeed, 

$\bullet$ if $t/\e \leq \tau$,
$|X_{t/\e}-\hX_{t/\e}|\leq |X_{t/\e}|+|X_{\tau+t/\e}-X_\tau|\leq |X_0|+\int_0^{t/\e} |V_s|\dd s+\int_{\tau}^{\tau+t/\e}
|V_s|\dd s \leq D^1$,

$\bullet$ if  $t/\e > \tau$, $|X_{t/\e}-\hX_{t/\e}|=|X_\tau + \hat X_{t/\e-\tau}-\hat X_{t/\e}|\leq|X_0|+
\int_0^{\tau} |V_s|\dd s + \int_{t/\e-\tau}^{t/\e} |\hat V_s|\dd s\leq D^1+D^{2,\e}_t$.

\noindent But $v^{(\beta)}_\e D^{1}$ a.s. tends to $0$, and using (b),
$$
\E[v^{(\beta)}_\e D^{2,\e}_t | \cF_\tau]
\leq \indiq_{\{t/\e\geq \tau\}} Cv^{(\beta)}_\e \int_{t/\e-\tau}^{t/\e} (1+s)^{1/(\beta+1)} \dd s 
\leq C \tau v^{(\beta)}_\e(1+t/\e)^{1/(\beta+1)},
$$
which a.s. tends to $0$ for all values of $\beta\geq 1$. Hence $v^{(\beta)}_\e D^{2,\e}_t$ tends to $0$ in probability.

\vip

We finally prove (d). First,  $(\e^{1/2}V_{t/\e})_{t\geq 0}\stackrel{d}\longrightarrow 
(U^{(1-\beta)}_t)_{t\geq 0}$ implies that $(\e^{1/2}V_{t/\e},\e^{3/2} X_{t/\e})_{t\geq 0}\stackrel{d}\longrightarrow 
(U^{(1-\beta)}_t,\int_0^t U^{(1-\beta)}_s\dd s)_{t\geq 0}$, simply because $\e^{3/2} X_{t/\e}=\e^{3/2}X_0 + 
\int_0^t (\e^{1/2}V_{s/\e})\dd s$.
We assume that this convergence holds true when $V_0=0$, consider any other solution $(V_t)_{\geq 0}$, 
introduce $\tau>0$
and $\hV_t=V_{\tau+t}$ as previously. Our goal is to check that 
$\Delta^\e_T=\e^{1/2}\sup_{[0,T]} |V_{t/\e}-\hV_{t/\e}| \to 0$ in probability.
We write $\Delta^\e_T\leq \Delta^{1,\e}+\Delta^{2,\e}_T$,
where $\Delta^{1,\e}=2\e^{1/2}\sup_{[0,2\tau]}|V_{s}|$ and $\Delta^{2,\e}_T=\sup_{[0,T]} \e^{1/2}
|\hV_{(t+\e \tau)/\e}-\hV_{t/\e}|$. Indeed,

$\bullet$ if $t\in [0,T]$ and 
$t/\e \leq \tau$, $\e^{1/2}|V_{t/\e}-\hV_{t/\e}| \leq  \e^{1/2}|V_{t/\e}|+\e^{1/2}|V_{\tau+t/\e}| \leq 
\Delta^{1,\e}$,

$\bullet$ if $t\in [0,T]$ and   $t/\e > \tau$, $\e^{1/2}|V_{t/\e}-\hV_{t/\e}|
=\e^{1/2}|\hV_{t/\e-\tau}- \hV_{t/\e}|\leq \Delta^{2,\e}_T$.

\noindent First, $\Delta^{1,\e}$ a.s. tends to $0$. Next, it is not hard to check that
$\Delta_T^{2,\e}$ goes in probability to $0$, using that $(\e^{1/2}\hV_{t/\e})_{t\geq 0}$
goes in law, in $C([0,\infty),\rr)$, to the continuous process $(U^{(1-\beta)}_t)_{t\geq 0}$.
\end{proof}

\subsection{The normal diffusion regime}
The following proof is standard, see e.g. Jacod-Shiryaev \cite[Chapter VIII, Section 3f]{js}.

\begin{proof}[Proof of Theorem \ref{mr}-(a).]
We assume that $\beta>5$ and, in view of Lemma \ref{ci}-(c), that $X_0=V_0=0$. Thanks to Lemma \ref{ci}-(a)
and since $\mu_\beta$ is a probability measure (because $\beta>1$), we classically deduce, see e.g. Kallenberg
\cite[Lemma 23.17 page 466 and Thm 23.14 page 464]{k}, that
\vip
(i) $V_t$ goes in law to $\mu_\beta$ as $t\to \infty$,
\vip
(ii) for all $\varphi \in L^1(\rr,\mu_\beta)$, $\lim_{t\to \infty} t^{-1}\int_0^t \varphi(V_s)\dd s
=\int_\rr \varphi \dd \mu_\beta$ a.s.
\vip

The function $g(v)=2 \int_0^v \Theta^{-\beta}(x)\int_x^\infty u\Theta^\beta(u)\dd u \dd x$ is odd
(since $\Theta$ is even and $\int_\rr u\Theta^\beta(u)\dd u \dd x=0$) and
solves the Poisson equation $g''(v)-\beta F(v)g'(v)=-2 v$, whence,
by the It\^o formula,
$$
g(V_t)=\intot g'(V_s)\dd B_s - \intot V_s\dd s,\quad i.e.\quad
X_t=\intot g'(V_s)\dd B_s - g(V_t).
$$
Consequently, we have $\e^{1/2}X_{t/\e}=M^\e_t- \e^{1/2}g(V_{t/\e})$,
where $M^\e_t=\e^{1/2}\int_0^{t/\e}g'(V_s)\dd B_s$.

\vip

For each $t\geq 0$, $\e^{1/2}g(V_{t/\e})$ tends to $0$ in probability: this follows from point (i) above. 
Here is why we deal with finite-dimensional distributions: it is not clear
that $\sup_{t\in [0,1]} |\e^{1/2}g(V_{t/\e})|$ tends to $0$.

\vip

We now show that $(M^\e_t)_{t\geq 0}$ goes in law (in the usual sense of continuous processes) 
to $(\sigma_\beta B_t)_{t\geq 0}$, and this will complete the proof.
It suffices, see e.g. Jacod-Shiryaev \cite[Theorem VIII-3.11 page 473]{js}, 
to verify that for each $t\geq 0$,
$\lim_{\e\to 0} \langle M^\e \rangle_t = \sigma_\beta^2 t$ in probability.
But $\langle M^\e \rangle_t= \e\int_0^{t/\e}[g'(V_s)]^2\dd s$,
which a.s. tends to $\sigma_\beta^2 t$ by point (ii). Indeed, using a symmetry argument,
$$
\int_\rr [g'(v)]^2 \mu_\beta(\dd v)  = 8 \int_0^\infty \!\! 
\Big[\Theta^{-\beta}(v)\int_v^\infty\!\! u \Theta^\beta(u)\dd u\Big]^2
\mu_\beta(\dd v) = 8c_\beta \int_0^\infty\!\! \Theta^{-\beta}(v)
\Big[\int_v^\infty\!\! u \Theta^\beta(u)\dd u \Big]^2\dd v=\sigma_\beta^2,
$$
recall Subsection \ref{mrc}. This value 
is finite, since $\Theta^{-\beta}(v)\Big[\int_v^\infty u \Theta^\beta(u)\dd u \Big]^2\sim 
(\beta-2)^{-2}v^{4-\beta}$
as $v \to \infty$ by \eqref{condi} and since $\beta>5$.
\end{proof}

\begin{rk}\label{aaa}
When $\beta=5$, our goal is to prove that 
$(\e^{1/2}|\log\e|^{-1/2}X_{t/\e})_{t\geq 0} \stackrel{f.d.}{\longrightarrow} 
(\sigma_5 B_t)_{t\geq 0}$. We can use exactly the same proof, provided we can show that for each $t\geq 0$,
in probability, as $\e\to 0$,
$$
\frac\e{|\log \e|} \int_0^{t/\e} [g'(V_s)]^2\dd s \longrightarrow \sigma_5^2 t.
$$
\end{rk}

\subsection{Scale function and speed measure}\label{ss3}

We introduce a few notation, closely linked with the scale function and speed measure of the process 
$(V_t)_{t\geq 0}$. This will allow us to rewrite $(X_t)_{t\geq 0}$
in a way that it resembles the objects appearing in Theorem \ref{ss} and Definition \ref{bs}.
All the functions below are defined on $\rr$. Recall our conditions on $\Theta$, see \eqref{condi}.

\vip

First, $h(v)=(\beta+1)\int_0^v [\Theta(u)]^{-\beta}\dd u$ is odd, increasing, bijective,
solves $h''=\beta F h'$, and we have 
$h(v)\stackrel{|v|\to\infty}\sim\sg(v)|v|^{\beta+1}$ and $h^{-1}(z)\stackrel{|z|\to\infty}\sim\sg(z)|z|^{1/(\beta+1)}$.

\vip

Next, $\sigma(z)=h'(h^{-1}(z))$ is even, bounded below by some $c>0$ and
$\sigma(z)\stackrel{|z|\to\infty}\sim (\beta+1)|z|^{\beta/(\beta+1)}$.

\vip

The function $\phi(z)=h^{-1}(z)/\sigma^2(z)$ is odd and $\phi(z)\stackrel{|z|\to\infty}\sim 
(\beta+1)^{-2}\sg(z)|z|^{(1-2\beta)/(\beta+1)}$.

\vip

When $\beta=5$, $\psi(z)=[g'(h^{-1}(z))]^2/\sigma^2(z)$ is even,
bounded and $\psi(z)\stackrel{|z|\to\infty}\sim 1/(81|z|)$. The even function 
$g'(v)=2\Theta^{-5}(v)\int_v^\infty u\Theta^5(u)\dd u\stackrel{|v|\to\infty}\sim 2|v|^2/3$ 
was introduced in the proof of Theorem \ref{mr}-(a).

\begin{lem}\label{ex}
Fix $\beta>0$, $\e>0$ and $a_\e>0$. 
Consider a Brownian motion $(W_t)_{t\geq 0}$. Define $A^{\e}_t=\e a_\e^{-2} \intot [\sigma(W_s/a_\e)]^{-2} \dd s$
and its inverse $(\tau^{\e}_t)_{t\geq 0}$,
which is a continuous increasing bijection
from $[0,\infty)$ into itself. Set 
$$
V^{\e}_t=h^{-1}(W_{\tau^{\e}_t}/a_\e) \quad \hbox{and} \quad 
X^\e_t=H^\e_{\tau^\e_t} \quad \hbox{where} \quad  H^\e_t=a_\e^{-2}\intot \phi(W_s/a_\e) \dd s.
$$
For $(V_t,X_t)_{t\geq 0}$ the unique solution of \eqref{eds} starting from $(0,0)$, we have
$(V_{t/\e},X_{t/\e})_{t\geq 0} \stackrel{d}=(V^{\e}_t,X^\e_t)_{t\geq 0}$.
\end{lem}

The result holds for any value of $a_\e>0$, but in each situation, we will choose it
judiciously, in such a way that $(A^\e_t)_{t\geq 0}$ a.s. converges, as $\e\to 0$, to the desired
limit time-change.

\begin{proof}
Since $\sigma$ is bounded below, $t\mapsto A^{\e}_t$ is a.s. continuous and strictly 
increasing. By recurrence of the Brownian motion, we also have $A^{\e}_\infty=\infty$ a.s. Hence 
$\tau^{\e}_t$ is well-defined, continuous, bijective from $[0,\infty)\mapsto[0,\infty)$ and
$Y^\e_t=W_{\tau^\e_t}$ classically solves, see e.g. Revuz-Yor \cite[Proposition 1.13 page 373]{ry}, 
$Y^\e_t= \e^{-1/2}a_\e\intot 
\sigma(Y^\e_s/a_\e) \dd B^\e_s$, for some Brownian motion $(B^\e_t)_{t\geq 0}$.
We then use the It\^o formula to write $V^\e_t=h^{-1}(Y^\e_t/a_\e)$ as
$$
V^\e_t = a_\e^{-1} \intot (h^{-1})'(Y^\e_s/a_\e)\e^{-1/2}a_\e\sigma(Y^\e_s/a_\e) \dd B^\e_s
+ \frac 1 2 a_\e^{-2}\intot (h^{-1})''(Y^\e_s/a_\e)\e^{-1}a_\e^2\sigma^2(Y^\e_s/a_\e) \dd s.
$$
But $(h^{-1})'(y)\sigma(y)=1$ and $(h^{-1})''(y)\sigma^2(y)=-\sigma'(y)=
- h''(h^{-1}(y))/h'(h^{-1}(y))=
-\beta F(h^{-1}(y))$, whence finally
$$
V^\e_t= \e^{-1/2}B^\e_t -  \frac\beta 2\e^{-1}\intot F(h^{-1}(Y^\e_s/a_\e))\dd s=\e^{-1/2}B^\e_t - \frac\beta 2
\e^{-1}\intot F(V^\e_s)\dd s.
$$
Starting from \eqref{eds}, we find
$$
V_{t/\e}=B_{t/\e} - \frac\beta 2\int_0^{t/\e}F(V_{s})\dd s=\e^{-1/2} (\e^{1/2}B_{t/\e}) -  
\frac\beta 2\e^{-1}\int_0^{t}F(V_{s/\e})\dd s.
$$
Hence $(V^\e_t)_{t\geq 0}$ and  $(V_{t/\e})_{t\geq 0}$ are two solutions of the same well-posed SDE,
driven by different Brownian motions, namely $(B^\e_t)_{t\geq 0}$ and $(\e^{1/2}B_{t/\e})_{t\geq 0}$.
They thus have the same law. 

\vip

Since $X_{t/\e}=\int_0^{t/\e}V_s\dd s=\e^{-1}\int_0^t V_{s/\e}\dd s$, we conclude that
$(V_{t/\e},X_{t/\e})_{t\geq 0}\stackrel{d}=(V^\e_t,\e^{-1}\int_0^t V^\e_{s}\dd s)_{t\geq 0}$. But
using the substitution $u=\tau^\e_s$, i.e. $s=A^\e_u$, whence 
$\dd s =\e a_\e^{-2}[\sigma(W_u/a_\e)]^{-2}\dd u$, we find
$$
\e^{-1}\int_0^t V^\e_{s}\dd s=\e^{-1}\intot h^{-1}(W_{\tau^{\e}_s}/a_\e)\dd s
=a_\e^{-2}\int_0^{\tau^\e_t} \frac{h^{-1}(W_u/a_\e)}{[\sigma(W_u/a_\e)]^{2}}\dd u
=a_\e^{-2}\int_0^{\tau^\e_t} \phi(W_u/a_\e)\dd u,
$$
which equals $H^\e_{\tau^\e_t}$ as desired.
\end{proof}

\subsection{Inverting time-changes}\label{ss4}

We recall the following classical and elementary results.

\begin{lem}\label{tc}
Consider, for each $n\geq 1$, a continuous increasing bijective function $(a^n_t)_{t\geq 0}$ from $[0,\infty)$
into itself, as well as its inverse $(r^n_t)_{t\geq 0}$. 

\vip

(a) Assume that $(a^n_t)_{t\geq 0}$ converges pointwise to some function $(a_t)_{t\geq 0}$ such that 
$\lim_{\infty} a_t=\infty$, denote by $r_t=\inf\{u\geq 0 : a_u>t\}$
its right-continuous generalized inverse and set $J=\{s\in [0,\infty) : r_{t-}<r_t\}$.
For all $t \in [0,\infty)\setminus J$, we have $\lim_{t\to \infty} r^n_t=r_t$.

\vip

(b) If  $(a^n_t)_{t\geq 0}$ converges (locally) uniformly to some strictly increasing function 
$(a_t)_{t\geq 0}$ such that $\lim_{\infty} a_t=\infty$, 
then $(r^n_t)_{t\geq 0}$ converges (locally) uniformly to $(r_t)_{t\geq 0}$,
the (classical) inverse of $(a_t)_{t\geq 0}$.
\end{lem}

\subsection{The integrated Bessel regime}\label{ss5}

We can now give the

\begin{proof}[Proof of Theorem \ref{mr}-(e)] Let $\beta \in (0,1)$ be fixed.
We consider a Brownian motion $(W_t)_{t\geq 0}$ and, as in Definition \ref{bs} with $\delta=1-\beta$,
we introduce the continuous strictly increasing bijective time-change 
$\bA_t=(\beta+1)^{-2}\intot |W_s|^{-2\beta/(\beta+1)}\dd s$, its inverse
$(\btau_t)_{t\geq 0}$ and the process $U^{(1-\beta)}_t=\sg(W_{\btau_t})|W_{\btau_t}|^{1/(\beta+1)}$.

\vip

We now apply Lemma \ref{ex} with the choice $a_\e=\e^{(\beta+1)/2}$: 
with the same Brownian motion as above, we consider, for each $\e>0$,
the time-change $A^\e_t=\e^{-\beta} \intot [\sigma(W_s/\e^{(\beta+1)/2})]^{-2}\dd s$, its inverse $\tau^\e_t$,
and $V^\e_t=h^{-1}(W_{\tau^{\e}_t}/\e^{(\beta+1)/2})$.
Recalling Lemma \ref{ci}-(d) and that
$(\e^{1/2}V_{t/\e})_{t\geq 0}\stackrel{d}=(\e^{1/2}V^\e_t)_{t\geq 0}$,
it suffices to prove that a.s., for all $T\geq 0$,
$\lim_{\e\to 0}  \sup_{[0,T]}|\e^{1/2}V^\e_t-U^{(1-\beta)}_t|=0$.

\vip

Since $\sigma(z)\geq c>0$ and $\sigma(z)\stackrel{|z|\to\infty}\sim (\beta+1)|z|^{\beta/(\beta+1)}$,
whence $\sigma^{-2}(z) \leq C |z|^ {-2\beta/(\beta+1)}$,
one has 
$$
\lim_{\e\to 0}\sup_{[0,T]} |A^\e_t-\bA_t|\leq \lim_{\e\to 0}\int_0^T \Big|\e^{-\beta}[\sigma(W_s/\e^{(\beta+1)/2})]^{-2}-
[(\beta+1)|W_s|^{\beta/(\beta+1)}]^{-2}\Big|\dd s=0 \quad \hbox{a.s.}
$$
by dominated convergence. Indeed, we have
$\sup_{\e>0}\e^{-\beta}[\sigma(W_s/\e^{(\beta+1)/2})]^{-2} \leq C |W_s|^{-2\beta/(\beta+1)}$,
and $\int_0^T |W_s|^{-2\beta/(\beta+1)}\dd s <\infty$ a.s.
because $2\beta/(\beta+1)<1$. 

\vip

By Lemma \ref{tc}-(b), we deduce that
for all $T\geq 0$, $\lim_{\e\to 0} \sup_{[0,T]}|\tau^\e_t - \btau_t|=0$ a.s. whence, by continuity of 
$(W_t)_{t\geq 0}$, 
\begin{equation}\label{uppp}
\lim_{\e\to 0} \sup_{[0,T]}|W_{\tau^\e_t} - W_{\btau_t}|=0 \quad \hbox{a.s. for all $T>0$.}
\end{equation}

We next claim that for all $M>0$,
$$
\kappa_\e(M)=\sup_{|z|\leq M} |\e^{1/2}h^{-1}(z/\e^{(\beta+1)/2}) - \sg(z)|z|^{1/(\beta+1)}| \to 0.
$$
Indeed, $h^{-1}$ being $C^1$, with $h^{-1}(0)=0$ and $h^{-1}(z) \stackrel{|z|\to\infty}\sim \sg(z)|z|^{1/(\beta+1)}$
the function $\gamma$ defined by 
$\gamma(z)=h^{-1}(z)/[ \sg(z)|z|^{1/(\beta+1)}]-1$ (and $\gamma(0)=-1$) 
is continuous, and 
$\lim_{|z|\to \infty} \gamma(z)=0$. Hence, 
$$
\kappa_\e(M) =  \sup_{|z|\leq M}|z|^{1/(\beta+1)}|\gamma(z/\e^{(\beta+1)/2})|
\leq \e^{1/4}||\gamma||_\infty + M^{1/(\beta+1)} \sup_{|z|\geq \e^{(\beta+1)/4}}|\gamma(z/\e^{(\beta+1)/2})|,
$$
which equals $\e^{1/4}||\gamma||_\infty + M^{1/(\beta+1)}\sup_{|z|\geq\e^{-(\beta+1)/4}}|\gamma(z)| \to 0$.

\vip

All in all, denoting by $M_T=\sup_{[0,T]}\sup_{\e\in(0,1)} |W_{\tau^\e_t}|$, which is a.s. finite by \eqref{uppp},
\begin{align*}
\sup_{[0,T]}|\e^{1/2}V^\e_t-U^{(1-\beta)}_t|=&\sup_{[0,T]}|\e^{1/2}h(W_{\tau^\e_t}/\e^{(\beta+1)/2}) - 
\sg(W_{\btau_t})|W_{\btau_t}|^{1/(\beta+1)}| \\
\leq & \kappa_\e(M_T) 
+ \sup_{[0,T]}\Big|\sg(W_{\tau^\e_t})|W_{\tau^\e_t}|^{1/(\beta+1)} - \sg(W_{\btau_t})|W_{\btau_t}|^{1/(\beta+1)}\Big|\to 0
\end{align*}
a.s., by \eqref{uppp} again. The proof is complete.
\end{proof}

\subsection{The L\'evy regime and the critical cases}\label{ss6}

We start with the following crucial lemma.

\begin{lem}\label{tloc}
Fix $\beta \in [1,5]$ and a Brownian motion $(W_t)_{t\geq 0}$, denote by $(L^0_t)_{t\geq 0}$ its local time at $0$
and by $(K_t)_{t\geq 0}$ the process defined in Theorem \ref{ss} with $\alpha=(\beta+1)/3$.
For each $\e>0$, consider the processes $(A^{\e}_t)_{t\geq 0}$
and $(H^\e_t)_{t\geq 0}$ built in Lemma \ref{ex} with the choice 
$a_\e=\e/[(\beta+1)c_\beta]$ if $\beta\in(1,5]$ and $a_\e=\e |\log \e|/2$ if $\beta=1$,
and with the same Brownian motion $(W_t)_{t\geq 0}$ as above.

\vip

(a) We always have $\lim_{\e\to 0}\sup_{[0,T]} |A^\e_t - L^0_t|=0$ a.s. for all $T>0$.

\vip

(b) If $\beta \in (1,5)$, $\lim_{\e\to 0}\sup_{[0,T]} |\e^{1/\alpha}H^\e_t - (\beta+1)^{1/\alpha-2}
c_\beta^{1/\alpha}K_t|=0$ a.s. for all $T>0$.

\vip

(c) If $\beta =1$, $\lim_{\e\to 0}\sup_{[0,T]} ||\e\log\e |^{3/2}H^\e_t - K_t/\sqrt 2|=0$ a.s. for all $T>0$.

\vip

(d) If $\beta=5$, $\lim_{\e\to 0} \sup_{[0,T]} |T^\e_t - \sigma_5^2 L^0_t|=0$ a.s. for all $T>0$,
with $T^\e_t= \frac \e{a_\e^2|\log\e |}\intot \psi(W_s/a_\e)\dd s$.
\end{lem}

We recall that $c_\beta=1/[\int_\rr [\Theta(v)]^\beta \dd v]$ (when $\beta>1$), that 
$\sigma_5^2=4c_5/27$ that the functions $h$, $\sigma$, $\phi$ and $\psi$
were introduced at the beginning of Subsection \ref{ss3}.

\begin{proof}
We start with (a) when $\beta>1$. We set $\gamma=(\beta+1)c_\beta$ and recall that $a_\e=\e/\gamma$,
whence $A^{\e}_t=\gamma^{2}\e^{-1}\intot [\sigma(\gamma W_s/\e)]^{-2} \dd s$.
Using the occupation times formula, see Revuz-Yor \cite[Corollary 1.6 page 224]{ry}, we may write
$$
A^{\e}_t=\int_\rr \frac{\gamma^2 L^x_t \dd x}{\e \sigma^2(\gamma x/\e)}=
\int_\rr \frac{\gamma L^{\e y/\gamma }_t \dd y}{\sigma^2(y)},
$$
where $(L^x_t)_{t\geq 0}$ is the local time of $(W_t)_{t\geq 0}$ at $x$.
Observe now that 
$$
\int_\rr \frac{\gamma\dd y}{\sigma^{2}(y)}= \int_\rr \frac{\gamma\dd y}{[h'(h^{-1}(y))]^{2}}
=\int_\rr \frac{\gamma\dd v}{h'(v)}=\int_\rr \frac{\gamma\Theta^\beta(v)\dd v}{(\beta+1)}= 1.
$$
Consequently, 
$$
\sup_{[0,T]} |A^\e_t - L^0_t| \leq 
\gamma \int_\rr\frac{\sup_{[0,T]}|L^{\e y/\gamma }_t-L^0_t| \dd y}{\sigma^2(y)},
$$
which a.s. tends to $0$ as $\e\to 0$ 
by dominated convergence, since $\sup_{[0,T]}|L^{\e y/\gamma}_t-L^0_t|$ a.s. tends to $0$
for each fixed $y$ by \cite[Corollary 1.8 page 226]{ry} and since $\sup_{[0,T]\times \rr} L^{x}_t$
is a.s. finite.

\vip

We now turn to point (a) when $\beta=1$, so that $a_\e=\e |\log \e|/2$. Also, $\sigma(z)$ is bounded below
and $\sigma(z)\stackrel{|z|\to\infty}\sim 2|z|^{1/2}$, from which
\begin{equation}\label{eqlog}
\int_{-x}^x \frac{\dd z} {\sigma^2(z) } \stackrel{x\to \infty}\sim \frac{\log x}2.
\end{equation}
We now fix $\delta>0$ and write $A^\e_t=I^{\e,\delta}_t+J^{\e,\delta}_t$, where
$$
I^{\e,\delta}_t=\intot \frac{\e \dd s}{ a_\e^2 \sigma^2(W_s/a_\e)}\indiq_{\{|W_s|\leq \delta\}} \quad \hbox{and}\quad 
J^{\e,\delta}_t=\intot \frac{\e \dd s}{ a_\e^2 \sigma^2(W_s/a_\e)}\indiq_{\{|W_s|> \delta\}}.
$$
There is $c>0$ such that $\sigma^2(z) \geq c(1+|z|)$, from which one verifies, using only that 
$|W_s|> \delta$ implies 
$\sigma^2(W_s/a_\e)\geq c(1+\delta/a_\e)\geq c \delta/a_\e$, that
$\sup_{[0,T]}|J^{\e,\delta}_t|\leq T \e/(c a_\e \delta)$, which tends to $0$ as $\e\to 0$.
We next use the occupation times formula to write
$$
I^{\e,\delta}_t= \int_{-\delta}^\delta \frac{\e L^x_t \dd x}{ a_\e^2 \sigma^2(x/a_\e)}=
\Big(\int_{-\delta}^\delta \frac{\e \dd x}{ a_\e^2 \sigma^2(x/a_\e)}\Big)L_t^0+
\int_{-\delta}^\delta \frac{\e (L^x_t-L^0_t) \dd x}{ a_\e^2 \sigma^2(x/a_\e)}=
r_{\e,\delta} L_t^0+ R^{\e,\delta}_t,
$$
the last identity standing for a definition.
But a substitution and \eqref{eqlog} allow us to write  
$$
r_{\e,\delta}=\int_{-\delta/{a_\e}}^{\delta/{a_\e}} \frac{\e \dd y}{ a_\e \sigma^2(y)} \stackrel{\e\to 0}\sim
\frac{\e \log (\delta/a_\e)}{2 a_\e} \longrightarrow 1
$$
as $\e\to 0$ since $a_\e=\e |\log \e|/2$. Recalling that $A^\e_t=r_{\e,\delta} L_t^0
+ R^{\e,\delta}_t+ J^{\e,\delta}_t$, we have proved that a.s., 
$$
\hbox{for all $\delta>0$,}\quad
\limsup_{\e\to 0} \sup_{[0,T]} |A^\e_t - L^0_t|\leq \limsup_{\e\to 0} \sup_{[0,T]} |R^{\e,\delta}_t|.
$$
But $\sup_{[0,T]}|R^{\e,\delta}_t|
\leq r_{\e,\delta} \times \sup_{[0,T]\times[-\delta,\delta]}|L^x_t-L^0_t|$, which implies that
$\limsup_{\e\to 0} \sup_{[0,T]} |A^\e_t - L^0_t|\leq \sup_{[0,T]\times[-\delta,\delta]}|L^x_t-L^0_t|$ a.s.
Letting $\delta\to 0$, using \cite[Corollary 1.8 page 226]{ry}, completes the proof.

\vip

Point (d), where $\beta=5$, is very similar to point (a) when $\beta=1$ and we only sketch the proof.
We set $\gamma=6c_5$ and recall that $a_\e=\e/\gamma$.
Since $\psi$ is bounded on $\rr$ and satisfies $\psi(z)\stackrel{|z|\to\infty}\sim
|81 z|^{-1}$, 
\begin{equation}\label{eqlog2}
\int_{-x}^x \psi(z)\dd z \stackrel{x\to \infty}\sim \frac{2\log x}{81}.
\end{equation}
Proceeding as previously, we can show rigorously that, 
for any $\delta>0$, uniformly in $t\in[0,T]$,
$$
T^\e_t=  \intot \frac{\gamma^2\psi(\gamma W_s/\e)\dd s}{ \e |\log \e|}
\simeq  \intot \frac{\gamma^2\psi(\gamma W_s/\e)\dd s}{ \e |\log \e|}\indiq_{\{|W_s|\leq \delta\}}
= \int_{-\delta}^\delta \frac{\gamma^2 \psi(\gamma x/\e) L^x_t \dd x}{\e |\log \e|},
$$
whence
$$
T^\e_t \simeq \Big(\frac \gamma {|\log \e|}\int_{-\delta \gamma/\e}^{\delta \gamma /\e} \psi(x)\dd x\Big)
\Big(L^0_t \pm \sup_{[0,T]\times[-\delta,\delta]}|L^x_t-L^0_t| \Big)\simeq \frac{2\gamma}{81}\Big( L^0_t
\pm \sup_{[0,T]\times[-\delta,\delta]}|L^x_t-L^0_t|\Big)
$$
by \eqref{eqlog2}. We conclude by letting $\delta$ tend to $0$, since $2\gamma/81=4c_5/27=\sigma_5^2$.

\vip

We now check (b), where $\beta \in (1,5)$ and $a_\e=\e/\gamma$ with $\gamma=(\beta+1)c_\beta$.
Recall that $\alpha=(\beta+1)/3$, whence $1/\alpha-2=(1-2\beta)/(\beta+1)$.
First, recalling Theorem \ref{ss} and the occupation times formula, 
$$
K^\eta_t=\int_\rr \sg(x)|x|^{(1-2\beta)/(\beta+1)}\indiq_{\{|x|\geq \eta\}}
L^x_t \dd x=\int_\rr \sg(x)|x|^{(1-2\beta)/(\beta+1)}\indiq_{\{|x|\geq \eta\}}
(L^x_t-L^0_t\indiq_{\{|x|\leq 1\}}) \dd x
$$
by symmetry. But we know from 
\cite[Corollary 1.8 page 226]{ry} and the fact that $\sup_{[0,T]\times\rr} L^x_t$ is a.s. 
finite that for all $\theta \in (0,1/2)$, all $T>0$,
$$
M_{\theta,T}=\sup_{[0,T]\times \rr} (|x|\land 1)^{-\theta}|L^x_t-L^0_t\indiq_{\{|x|\leq 1\}}|<\infty \quad \hbox{a.s.}
$$
Since $(1-2\beta)/(\beta+1)>-3/2$ (because $\beta<5$), we deduce that
$(K^\eta_t)_{t\geq 0}$ a.s. converges uniformly on $[0,T]$, as $\eta\to 0$, to
$$
K_t=\int_\rr \sg(x)|x|^{(1-2\beta)/(\beta+1)}(L^x_t-L^0_t\indiq_{\{|x|\leq 1\}}) \dd x.
$$
Similarly, by oddness of $\phi$ (and since $\e^{1/\alpha}a_\e^{-2}=\gamma^2\e^{1/\alpha-2}=\gamma^2
\e^{(1-2\beta)/(\beta+1)}$),
$$
\e^{1/\alpha}H^\e_t= 
\gamma^2\e^{(1-2\beta)/(\beta+1)}\intot \phi(\gamma W_s/\e)\dd s= 
\int_\rr \gamma^2 \e^{(1-2\beta)/(\beta+1)}\phi(\gamma x/\e)
(L^x_t-L^0_t\indiq_{\{|x|\leq 1\}}) \dd x.
$$
Hence
\begin{align*}
&\sup_{[0,T]}|\e^{1/\alpha}H^\e_t - (\beta+1)^{1/\alpha-2}c_\beta^{1/\alpha}K_t|\\
\leq& \int_\rr \Big|\gamma^2\e^{(1-2\beta)/(\beta+1)}\phi(\gamma x/\e)- 
(\beta+1)^{1/\alpha-2}c_\beta^{1/\alpha}\sg(x)|x|^{(1-2\beta)/(\beta+1)}
\Big| \sup_{[0,T]}|L^x_t-L^0_t\indiq_{\{|x|\leq 1\}}|\dd x\\
\leq& M_{\theta,T} \int_\rr \Big|\gamma^2\e^{(1-2\beta)/(\beta+1)}\phi(\gamma x/\e)
- (\beta+1)^{1/\alpha-2}c_\beta^{1/\alpha}\sg(x)|x|^{(1-2\beta)/(\beta+1)}
\Big| (|x| \land 1)^\theta \dd x
\end{align*}
for any $\theta\in (0,1/2)$. Using the equivalence $\phi(z)\stackrel{|z|\to\infty}\sim 
(\beta+1)^{-2}\sg(z)|z|^{(1-2\beta)/(\beta+1)}$, the bound $|\phi(z)| \leq C(1+|z|^{(1-2\beta)/(\beta+1)})$
and that $(1-2\beta)/(\beta+1)>-3/2$, we conclude, by dominated convergence, that
$\lim_{\e \to 0}\sup_{[0,T]}|\e^{1/\alpha}H^\e_t - (\beta+1)^{-2}K_t|=0$ a.s.
This uses that 
$$
\gamma^{2+(1-2\beta)/(\beta+1)}(\beta+1)^{-2}=(\beta+1)^{1/\alpha-2}c_\beta^{1/\alpha}.
$$

We finally check (c), where $\beta=1$ and $a_\e=\e|\log\e|/2$.
No principal value is needed here and it actually holds true that
$K_t=\lim_{\eta\to 0} \intot \sg(W_s)|W_s|^{-1/2}\indiq_{\{|W_s|\geq \eta\}}\dd s=\intot \sg(W_s)|W_s|^{-1/2}\dd s$.
Also, we have $|\e\log\e|^{3/2} H^\e_t = 4 |\e\log\e|^{-1/2} \intot \phi(2W_s/|\e\log\e|)\dd s$.
Using that $\phi(z)\stackrel{|z|\to \infty}
\sim\sg(z)|z|^{-1/2}/4$, that $|\phi(z)|\leq C |z|^{-1/2}$ and that
$\int_0^T |W_s|^{-1/2}\dd s <\infty$ a.s., 
one verifies, by dominated convergence, that
$$
\lim_{\e \to 0}\sup_{[0,T]} \Big||\e\log\e |^{3/2}H^\e_t - K_t/\sqrt 2\Big| \leq 
\lim_{\e\to 0}\int_0^T \Big|4|\e\log\e |^{-1/2}\phi(2W_s/|\e\log\e|)-\sg(W_s)|W_s|^{-1/2}/\sqrt 2\Big|\dd s=0
$$
a.s., as desired. Such a simple proof can also be handled to check (b) when $\alpha<1$, i.e. $\beta<2$.
\end{proof}

We now give the

\begin{proof}[Proof of Theorem \ref{mr}-(b)-(c)-(d)]
Fix $\beta \in [1,5]$ and a Brownian motion $(W_t)_{t\geq 0}$, denote by $(L^0_t)_{t\geq 0}$ its local time
and by $\tau_t=\inf\{u\geq 0 : L^0_u> t\}$. Consider the process $(K_t)_{t\geq 0}$ defined in Theorem \ref{ss}
with $\alpha=(\beta+1)/3$.
For each $\e>0$, consider the processes $(A^{\e}_t)_{t\geq 0}$,
$(\tau^\e_t)_{t\geq 0}$, $(V^\e_t)_{t\geq 0}$ and $(H^\e_t)_{t\geq 0}$ built in Lemma \ref{ex} 
with the choice $a_\e=\e/[(\beta+1)c_\beta]$ if $\beta\in(1,5]$ and $a_\e=\e |\log \e|/2$ if $\beta=1$.

\vip

{\it Point (b): $\beta=5$.} As seen in Remark \ref{aaa}, we only have to verify that
$\e|\log \e|^{-1} \int_0^{t/\e} [g'(V_s)]^2\dd s$,
which equals  $|\log \e|^{-1} \int_0^{t} [g'(V_{s/\e})]^2\dd s$, goes in probability to $\sigma_5^2 t$
for each $t\geq 0$.
By Lemma \ref{ex}, we may equivalently show that $J^\e_t=|\log \e|^{-1} \int_0^{t} [g'(V_s^\e)]^2\dd s 
\longrightarrow \sigma_5^2 t$. But as usual, 
$$
J^\e_t=\int_0^{t} \frac{[g'(h^{-1}(W_{\tau^\e_s}/a_\e))]^2}{|\log \e|}\dd s=
\int_0^{\tau^\e_t} \frac{\e [g'(h^{-1}(W_{u}/a_\e))]^2}{a_\e^2 |\log \e| [\sigma(W_u/a_\e)]^2} \dd u
=\frac \e {a_\e^2 |\log \e|} \int_0^{\tau^\e_t} \psi(W_{u}/a_\e)\dd u=T^\e_{\tau^\e_t}
$$ 
with the notation of Lemma \ref{tloc}-(e).
By Lemma \ref{tloc}-(a), $\sup_{[0,T]} |A^\e_t- L^0_t|\to 0$ a.s. Since $(\tau_{t})_{t\geq 0}$,
the generalized inverse of $(L^0_t)_{t\geq 0}$, 
has no fixed times of jump, we deduce from Lemma \ref{tc}-(a) that for each $t\geq 0$,
$\tau^\e_t \to \tau_{t}$ a.s. Using now Lemma \ref{tloc}-(d), 
$\sup_{[0,T]} |T^\e_t-\sigma_5^2L^0_t|\to 0$. All in all, for $t\geq 0$ fixed,
$$
|J^\e_t-\sigma_5^2t|\leq |T^\e_{\tau^\e_t}-  \sigma_5^2L^0_{\tau^\e_t}| + |\sigma_5^2L^0_{\tau^\e_t}-\sigma_5^2L^0_{\tau_t}|
+|\sigma_5^2L^0_{\tau_{t}}-\sigma_5^2 t|
$$
which a.s. tends to $0$: for the first term, we use that $\sup_{[0,T]} |T^\e_s-\sigma_5^2L^0_s|\to 0$ a.s.
and that $\sup_{\e\in(0,1)}\tau^\e_t$ is a.s. finite (since it has a finite limit as $\e\to 0$),
for the second one, we use that $(L^0_t)_{t\geq 0}$ is a.s. continuous and that $\tau^\e_t \to \tau_{t}$ a.s.,
for the last one, we use that (for $t\geq 0$ fixed) $L^0_{\tau_t}=t$ a.s.

\vip

{\it Point (c): $\beta\in(1,5)$.}
By Lemma \ref{ci}-(c), we may assume that $X_0=V_0=0$, whence, 
by Lemma \ref{ex}, $(X_{t/\e})_{t\geq 0}\stackrel{d}= (H^\e_{\tau^\e_t})_{t\geq 0}$. To verify that 
$(\e^{1/\alpha} X_{t/\e})_{t\geq 0}\stackrel{f.d.}\to (\sigma_\beta S^{(\alpha)}_t)_{t\geq 0}$, 
it is thus sufficient to verify that
for each $t\geq 0$ fixed, a.s.,
\begin{equation*}
\Delta_t(\e)=|\e^{1/\alpha}H^\e_{\tau^\e_t} - (\beta+1)^{1/\alpha-2}c_\beta^{1/\alpha}K_{\tau_{t}}|\to 0.
\end{equation*}
Indeed, Theorem \ref{ss} tells us that $S^{(\alpha)}_t=\sigma_\beta^{-1}(\beta+1)^{1/\alpha-2}c_\beta^{1/\alpha}K_{\tau_{t}}$ 
is a symmetric $\alpha$-stable process with $\E[\exp(i \xi S^{(\alpha)}_t)]
=\exp(-\kappa_\alpha t |\sigma_\beta^{-1}(\beta+1)^{1/\alpha-2}c_\beta^{1/\alpha}\xi|^\alpha)=\exp(- t |\xi|^\alpha)$ 
by definition of $c_\beta$ and $\kappa_\alpha$.

\vip

By Lemma \ref{tloc}-(a), $\sup_{[0,T]} |A^\e_t- L^0_t|\to 0$ a.s. Since $(\tau_{t})_{t\geq 0}$,
the generalized inverse of $(L^0_t)_{t\geq 0}$,
has no fixed time of jump, we deduce from Lemma \ref{tc}-(a) that
$|\tau^\e_t -\tau_{t}|\to 0$ a.s. By Lemma \ref{tloc}-(b), 
$\lim_{\e\to 0}\sup_{[0,T]} |\e^{1/\alpha}H^\e_t - (\beta+1)^{1/\alpha-2}c_\beta^{1/\alpha}K_t|=0$ a.s. for all $T>0$.
All in all,
$$
\Delta_t(\e)\leq |\e^{1/\alpha}H^\e_{\tau_t^\e} - (\beta+1)^{1/\alpha-2}c_\beta^{1/\alpha} K_{\tau_t^\e}|+
(\beta+1)^{1/\alpha-2}c_\beta^{1/\alpha}|K_{\tau_t^\e} -K_{\tau_t}|,
$$
which a.s. tends to $0$: for the first term, we use that $\lim_{\e\to 0}\sup_{[0,T]} |\e^{1/\alpha}H^\e_s 
- (\beta+1)^{1/\alpha-2}c_\beta^{1/\alpha}K_s|=0$ 
a.s. and that $\sup_{\e\in(0,1)}\tau^\e_t$ is a.s. finite (since it has a finite limit as $\e\to 0$),
for the second one, we use that $(K_t)_{t\geq 0}$ is a.s. continuous and that $\tau^\e_t \to \tau_{t}$ a.s.

\vip
{\it Point (d): $\beta=1$.} By Lemma \ref{ci}-(c), we may assume that $X_0=V_0=0$, whence, 
by Lemma \ref{ex}, $(X_{t/\e})_{t\geq 0}\stackrel{d}= (H^\e_{\tau^\e_t})_{t\geq 0}$. To verify that 
$(|\e\log\e|^{3/2} X_{t/\e})_{t\geq 0}\stackrel{f.d.}\to (\sigma_1 S^{(2/3)}_t)_{t\geq 0}$, 
it is thus sufficient to verify that
for each $t\geq 0$ fixed, a.s.,
\begin{equation*}
\Delta'_t(\e)=\Big||\e\log\e|^{3/2}H^\e_{\tau^\e_t} - K_{\tau_{t}}/\sqrt2\Big|\to 0.
\end{equation*}
Indeed, Theorem \ref{ss} tells us that $(S^{(2/3)}_t)_{t\geq 0}=(\sqrt2\sigma_1)^{-1}K_{\tau_{t}}$ 
is a symmetric $\alpha$-stable process with $\E[\exp(i \xi S^{(2/3)}_t)]
=\exp(-\kappa_{2/3} t |(\sqrt 2\sigma_1)^{-1}\xi|^{2/3})=\exp(- t |\xi|^{2/3})$  
by definition of $\sigma_1$ and $\kappa_{2/3}$. 

\vip

By Lemma \ref{tloc}-(a), $\sup_{[0,T]} |A^\e_t- L^0_t|\to 0$ a.s., whence, by Lemma \ref{tc}-(a),
$|\tau^\e_t -\tau_{t}|\to 0$ a.s. By Lemma \ref{tloc}-(d), 
$\lim_{\e\to 0}\sup_{[0,T]} ||\e\log\e|^{3/2}H^\e_t- K_t/\sqrt 2|=0$ a.s.
Thus
$$
\Delta_t'(\e)\leq \Big||\e\log\e|^{3/2}H^\e_{\tau^\e_t}- K_{\tau^\e_t}/\sqrt 2\Big|
+|K_{\tau^\e_t}-K_{\tau_{t}}|/\sqrt 2\to 0
$$
by continuity of $(K_t)_{t\geq 0}$.
\end{proof}

\subsection{Decoupling}

We end the paper with the

\begin{proof}[Proof of Corollary \ref{coco}]
Let $\beta>1$ be fixed, as well as the solution $(V_t,X_t)_{t\geq 0}$ to \eqref{eds}, starting from some
given initial condition $(V_0,X_0)$ and driven by some Brownian motion $(B_t)_{t\geq 0}$. 
We introduce $\cF_t=\sigma(X_0,V_0,B_s,s\leq t)$.
We know that $(v^{(\beta)}_\e X_{t/\e})_{t\geq 0}\stackrel{f.d.}\to 
(X^{(\beta)}_{t})_{t\geq 0}$, where $v^{(\beta)}_\e\to 0$ and $(X^{(\beta)}_{t})_{t\geq 0}$ are the rate and limit process
appearing in Theorem \ref{mr}. We fix $t>0$, $\varphi \in C^1_b(\rr)$ and  
$\psi \in B_b(\rr)$, and our goal is to verify that, setting $\mu_\beta(\psi)=\int_\rr \psi \dd \mu_\beta$,
as $\e\to 0$,
$$
\Delta_\e=\Big|\E[\varphi(v^{(\beta)}_\e X_{t/\e})\psi(V_{t/\e})] - \E[\varphi(X^{(\beta)}_{t})]\mu_\beta(\psi) \Big| \to 0.
$$

{\it Step 1.} We check here that for any $h\in (0,t)$, $\delta_\e=\E[|\E[\psi(V_{t/\e})|\cF_{(t-h)/\e}]
-\mu_\beta(\psi)|]\to 0$ as $\e\to 0$. We use the common notation notation  
$P_t\psi(v)=\E_v[\psi(V_t)]$ (although we still use $\E$ and $\PR$ without subscript when working
with the initial condition $V_0$). We
introduce the total variation norm  $|| \,\cdot\,||_{TV}$. 
By the Markov property,
$$
\delta_\e=\E[|P_{h/\e}\psi(V_{(t-h)/\e})-\mu_\beta(\psi)|]\leq \delta^1_\e+\delta^2_\e,
$$
where, introducing some $\mu_\beta$-distributed $\bar V$ such that $\PR(\bar V \neq V_{(t-h)/\e})=
||\cL(V_{(t-h)/\e}) - \mu_\beta||_{TV}$,
\begin{align*}
\delta^1_\e =& \E[|P_{h/\e}\psi(V_{(t-h)/\e})- P_{h/\e}\psi(\bar V)|]\leq 2||\psi||_\infty \PR(\bar V \neq V_{(t-h)/\e})=
2||\psi||_\infty||\cL(V_{(t-h)/\e}) - \mu_\beta||_{TV},\\
\delta^2_\e = & \E[|P_{h/\e}\psi(\bar V)-\mu_\beta(\psi)|].
\end{align*}
But, by  Lemma \ref{ci}-(a) and \cite[Lemma 23.17 page 466]{k}, 
$||\cL(V_s) - \mu_\beta||_{TV} \to 0$ (when starting from any initial condition $V_0$) 
as $s\to \infty$. Hence $\delta^1_\e$ tends to $0$ as $\e\to 0$.
We also have $\lim_{s\to \infty} P_s\psi(v)=\mu_\beta(\psi)$ for all $v\in \rr$, so that $\delta^2_\e$
tends to $0$ by dominated convergence.

\vip

{\it Step 2.} For any $h\in (0,t)$, we write $\Delta_\e \leq \Delta_{\e,h}^1+\Delta_{\e,h}^2+\Delta_{\e,h}^3
+\Delta_{\e,h}^4$, where
\begin{align*}
\Delta_{\e,h}^1 =& \Big|\E[\varphi(v^{(\beta)}_\e X_{t/\e})\psi(V_{t/\e})] 
- \E[\varphi(v^{(\beta)}_\e X_{(t-h)/\e})\psi(V_{t/\e})]\Big|,\\
\Delta_{\e,h}^2 =& \Big|\E[\varphi(v^{(\beta)}_\e X_{(t-h)/\e})\psi(V_{t/\e})]
- \E[\varphi(v^{(\beta)}_\e X_{(t-h)/\e})]\mu_\beta(\psi)\Big|,\\
\Delta_{\e,h}^3 =& \Big|\E[\varphi(v^{(\beta)}_\e X_{(t-h)/\e})]\mu_\beta(\psi) 
- \E[\varphi(X^{(\beta)}_{t-h})]\mu_\beta(\psi)\Big|,\\
\Delta_{\e,h}^4 =& \Big|\E[\varphi(X^{(\beta)}_{t-h})]\mu_\beta(\psi)-\E[\varphi(X^{(\beta)}_{t})]\mu_\beta(\psi)\Big|.
\end{align*}
By Theorem \ref{mr}, $\lim_{\e\to 0} \Delta_{\e,h}^3=0$ and, 
with $C=||\psi||_\infty(||\varphi||_\infty+||\varphi'||_\infty )$
$$
\limsup_{\e\to 0} \Delta_{\e,h}^1 \leq C
\limsup_{\e\to 0} \E[|v^{(\beta)}_\e X_{t/\e}-v^{(\beta)}_\e X_{(t-h)/\e}|\land 1]=C
\E[|X^{(\beta)}_{t}-X^{(\beta)}_{t-h}|\land 1].
$$
Also, $\Delta_{\e,h}^4 \leq C \E[|X^{(\beta)}_{t}-X^{(\beta)}_{t-h}|\land 1]$ and
$\Delta_{\e,h}^2\leq ||\varphi||_\infty \E[|\E[\psi(V_{t/\e})|\cF_{(t-h)/\e}]-\mu_\beta(\psi)|] \to 0$
by Step 1. All in all,
$\limsup_{\e\to 0} \Delta_\e \leq 2 C\E[|X^{(\beta)}_{t}-X^{(\beta)}_{t-h}|\land 1]$ 
for any $h\in (0,t)$. Letting $h \downarrow 0$ ends the proof.
\end{proof}

\end{document}